\documentclass[a4paper,12pt]{article}

\usepackage{amsmath}
\usepackage{amsfonts}
\usepackage{amssymb}
\usepackage{amstext    }
\usepackage{amsthm    }

\usepackage{mathrsfs}

\usepackage{hyperref}

\usepackage{color}

\usepackage[T1]{fontenc}
\usepackage{graphicx}
\usepackage[all,cmtip]{xy}

\usepackage{epsf}
\usepackage{psfrag}

\usepackage{fancyhdr}

\usepackage{url}

\newtheorem{theorem}{Theorem}[section]
\newtheorem{lemma}[theorem]{Lemma}
\newtheorem{definition}[theorem]{Definition}
\newtheorem{proposition}[theorem]{Proposition}
\newtheorem{remark}[theorem]{Remark}
\newtheorem{corollary}[theorem]{Corollary}

\newcommand{\cali}[1]{\mathscr{#1}}

\numberwithin{equation}{section}

\newcommand{\dist}{{\rm dist}}
\newcommand{\ddc}{{dd^c}}

\newcommand{\Jac}{{\rm Jac}}
\newcommand{\supp}{{\rm supp}}
\newcommand{\id}{{\rm Id}}

\newcommand{\Ac}{\cali{A}}

\newcommand{\Cc}{\cali{C}}

\newcommand{\Fc}{\cali{F}}

\newcommand{\Pc}{\cali{P}}

\newcommand{\Pb}{\mathbb{P}}
\newcommand{\proj}{\mathbb{P}}
\newcommand{\Cb}{\mathbb{C}}
\newcommand{\C}{\mathbb{C}}
\newcommand{\Nb}{\mathbb{N}}
\newcommand{\N}{\mathbb{N}}
\newcommand{\Zb}{\mathbb{Z}}

\usepackage[T1]{fontenc}
\usepackage[english]{babel}

\author{
  Sandrine Daurat   \thanks{ Sandrine Daurat  was partially supported by the FMJH (Governement Program:  ANR-10-CAMP-0151-02) and the ANR project LAMBDA (Governement Program: ANR-13-BS01-0002)}\\
  \and
 Johan Taflin\\
}
\title{Codimension one attracting sets in $\Pb^k(\Cb)$}
\date{}
\begin{document}
\maketitle
\begin{abstract}
We are interested in attracting sets of $\Pb^k(\Cb)$ which are of small topological degree and of codimension $1.$ We first show that there exists a large family of examples. Then we study their ergodic and pluripotential theoretic properties.
\end{abstract}

\section{Introduction}
The aim of this article is to study a special class of attracting sets for holomorphic endomorphisms of the complex projective space $\Pb^k.$ Let $f$ be such an endomorphism, of algebraic degree $d\geq2.$ Recall that an \textit{attracting set} of $f$ is a compact set $\Ac\subset\Pb^k$ admiting an open neighborhood $U$ which is a trapping region i.e. $f(U)\Subset U$ and $\Ac:=\bigcap_{n\geq0}f^n(U).$ These sets play an important role in dynamical systems and they have been studied by many authors, see e.g.  \cite{for:sib, for:wei, jon:wei, d-attractor,  ron} 
 for the complex setting.

In \cite{daurat}, the first author introduced a notion of \textit{dimension} for attracting sets and the notion of attracting sets of \textit{small topological degree}, see Definition \ref{def dim} and Definition \ref{def std} below. In this  framework several interesting dynamical and geometric properties have been pointed out. Among other things, it was shown in this paper that such attracting sets are non-algebraic and also that such examples are quite abundant in $\Pb^2.$ Our purpose here is to answer questions that were left open.

Our first result is a generalization of the construction \cite{daurat} to arbitrary dimensions
 $k\geq 2.$ To be more specific, let $k\geq2$ be an integer and for $d\geq2,$ let $\Fc_d$ denote the set of pairs $f=(f_\infty,R)$ where $R$ is an homogeneous polynomial of degree $d$ in $\Cb^{k}$ and $f_\infty=(F_0,\ldots,F_{k-1})$ is a $k$-tuple of homogeneous polynomials which defines a holomorphic endomorphism of $\Pb^{k-1}$ of degree $d.$ The set $\Fc_d$ is a quasi-projective variety. If $f=(f_\infty,R)\in\Fc_d$ and $\varepsilon\in\Cb$ then we denote by $f_\varepsilon$ the endomorphism of $\Pb^k$ defined by
$$[z_0:\cdots:z_{k-1}:z_k]=[z':z_k]\mapsto[f_\infty(z'):z_k^d+\varepsilon R(z')].$$
When $\varepsilon=0,$ these maps admit the hyperplane $\{z_k=0\}$ as an attracting set, which is algebraic and thus cannot be of small topological degree. However, a generic perturbation $f_\varepsilon$ has an attracting set with this property.
\begin{theorem}\label{th ex generic}
There exists a non empty Zariski open set $\Omega\subset \Fc_d$, such that if $f\in \Omega$ then for $\varepsilon\in\Cb$ with $|\varepsilon|\neq 0$ small enough, the maps $f_\varepsilon$ is of small topological degree on an attracting set $\Ac_\varepsilon$ close to the hyperplane $\{z_k=0\}.$ 
\end{theorem}

The idea of the construction of the Zariski open set $\Omega$ is the same as in \cite{daurat}.
In particular, the proof uses the fact 
that the maps $f_\varepsilon$ preserve the pencil of lines passing through $[0:\cdots:0:1].$ The new delicate part here is to provide examples in any dimension and any degree to ensure that $\Omega$ is not empty.

Notice that the property to possess an attracting set of small topological degree is stable under small perturbations. It is also the case for all our other assumptions below. Hence, Theorem \ref{th ex generic} provides a lot of attracting sets of small topological degree
  of codimension $1,$ i.e. they carry non-trivial positive closed currents of bidegree $(1,1)$ and they are not equal to the whole space $\Pb^k.$
   It would be interesting to understand the nature of the open set in the whole parameter space where the small topological degree condition holds.

    Our other results are related to the ergodic and pluripotential theoretic properties of such attracting sets.
Let $\Ac$ be an attracting set of codimension $1$ with a trapping region $U.$ Let $\Cc_1(U)$ be the set of all positive closed $(1,1)$-currents $S$ supported in $U$ and of mass $1.$ Under a mild assumption on $U,$ which will be referred to as (HD), Dinh constructed in \cite{d-attractor} an \textit{attracting current} $\tau\in\Cc_1(U)$ such that any continuous form $S\in\Cc_1(U)$ satisfies
$$\lim_{n\to+\infty}\Lambda^nS=\tau,$$
where $\Lambda:=d^{1-k}f_*$ is the normalized push forward operator acting on $\Cc_1(U).$ In particular $\tau$ is invariant under $\Lambda.$ He also obtained that if $T$ denotes the Green current of $f$ then the measure $\nu:=\tau\wedge T^{k-1}$ is mixing and of entropy $(k-1)\log d.$

The first author showed that when $f$ is of small topological degree on $\Ac$ then $\tau$ has bounded potential and asked whether this potential was continuous. The answer is positive, provided that the geometric assumption (HD) holds.
\begin{theorem}\label{th pot conti}
Let $f$ be a holomorphic endomorphism of $\Pb^k.$ If $f$ is of small topological degree on an attracting set $\Ac$ of codimension $1$  satisfying (HD) then the attracting current $\tau$ on $\Ac$ has continuous potential.
\end{theorem}
 More precisely, we are going to prove that there exist constants $A>0$ and $\alpha>0$ such that for all $x,y\in \Pb^k$
$|u(x)-u(y)|\leq A \left|\log\dist(x,y)\right|^{-\alpha}$
where $u$ is such that $\tau-\omega=\ddc u$.

This result gives information on the size of the attracting set in the sense that the more regular the potential of $\tau$ is, the bigger its support is. However, we were not able to obtain Hölder continuity which would give a lower bound on the Hausdorff dimension of $\Ac.$

Another interesting question about currents in $\Cc_1(U)$ is whether $\tau$ is the unique current in $\Cc_1(U)$ invariant by $\Lambda.$ Actually, under the same assumptions, this is indeed the case and moreover we obtain the following equidistribution result towards $\tau$ with uniform exponential speed.
\begin{theorem}\label{th speed}
Let $f$ and $\Ac$ be as in Theorem \ref{th pot conti}. There exist constants $0<\delta<1$ and $c\geq0$ such that
$$|\langle \Lambda^nS-\tau,\phi\rangle|\leq c\delta^n\|\phi\|_{\mathcal C^2},$$
for all $S\in\Cc_1(U)$ and test forms $\phi$ of class $\mathcal C^2.$ In particular, $\tau$ is the unique current in $\Cc_1(U)$ invariant by $\Lambda.$
\end{theorem}
Both theorems can be reformulated as speed of convergence results for potentials, in the sup norm for the first one and in the $L^1$ norm for the second one. Besides the assumption of small topological degree, the proofs are based on the ``geometric'' properties of the space $\Cc_1(U)$ which come from (HD).\\

Our last result is about the hyperbolicity of certain invariant measures.  In \cite{deT:expo}, de Thélin obtained a powerful result which relates the entropy of an invariant measure, its Lyapunov exponents and the dynamical degrees of $f.$ For an endomorphism of $\Pb^k,$ the dynamical degree of order $i$ is just $d^i.$ The last one, $d^k,$ is the topological degree of $f$ on $\Pb^k.$ We will substitute it by a  local version to prove
the following result.
\begin{theorem}\label{th hyp}
Let $\nu$ be an ergodic measure and of entropy $h(\nu)=(k-1)\log d.$ 
If there exists an open neighbourhood $V$ of $\supp(\nu)$ such that 
$$d_{k,loc}:=\limsup_{n\to+\infty}\left(\int_V(f^n)^*(\omega^k)\right)^{1/n}<d^{k-1},$$
then $\nu$ is hyperbolic and its Lyapunov exponents satisfy
$$\chi_1\geq\cdots\geq\chi_{k-1}\geq \frac{1}{2}\log d>0>\frac{1}{2}\log\left(\frac{d_t}{d^{k-1}}\right)\geq \chi_k\geq-\infty.$$
\end{theorem}

The assumption $d_{k,loc}<d^{k-1}$ holds in all examples known by the authors. Moreover, to be of small topological degree on $\Ac$ is much stronger than $d_{k,loc}<d^{k-1}.$ The first one is a uniform pointwise property while the latter deals with an average on $V.$

Hence, Theorem \ref{th hyp} applies to the measure $\nu$ constructed by Dinh, as long as $d_{k,loc}<d^{k-1},$ so in particular for all the examples obtained in Theorem \ref{th ex generic}.

To conclude, let us underline here that  results similar to Theorem \ref{th speed} and Theorem \ref{th hyp} have been obtained by the second author in \cite{attrac-speed} in a different setting well adapted to perturbations of super-attractive attracting sets of  any codimension, which are all the know examples.
 However, this set of examples is
 quite reduced and we believe these theorems, or at least there proofs, could be applied to a very large class of codimension $1$ attracting sets. Notice furthermore that, in Theorem \ref{th hyp}, the support of  the measure $\nu$ is not necessary contained in an attracting set.

The paper is organized as follows. In Section \ref{sec as}, we recall Dinh's setting (HD) and the notions of dimension and of small topological degree attached to attracting sets. Section \ref{sec ex} is devoted to the proof of Theorem \ref{th ex generic} on the existence of examples in any dimension. Then, in Section \ref{sec pot} we prove Theorem \ref{th pot conti} and Theorem \ref{th speed} about the attracting current $\tau.$ Finally, we discuss in Section \ref{section expo str négatif} the hyperbolicity of measures of entropy $(k-1)\log d.$

\bigskip

We thank H. de Thélin for having pointed out to us that we should be able to extend his result \cite{deT:expo} to our case.
\section{Attracting sets in $\Pb^k$}\label{sec as}
In this section, we summarize several notions and results about attracting sets.

For $0\leq p\leq k$ and $E\subset \Pb^k,$ let denote by $\Cc_p(E)$ the set of all positive closed $(p,p)$-currents of mass $1$ and supported on $E.$ Here, the mass of a $(p,p)$-current $S$ is $\|S\|:=\langle S,\omega^{k-p}\rangle$ where $\omega$ is the standard Fubini-Study form on $\Pb^k.$

Following \cite{daurat}, we define the dimension of an attracting set.

\begin{definition}\label{def dim}
Let $\Ac$ be an attracting set. We say that $\Ac$ has dimension $k-p,$ or equivalently codimension $p,$ if $\Cc_p(\Ac)\neq\varnothing$ and $\Cc_{p-1}(\Ac)=\varnothing.$
\end{definition}
The whole space $\Pb^k$ is always an attracting set of maximal dimension $k$ and it is the only one. In the same way, the attracting periodic points  correspond exactly to attracting sets of dimension $0.$ Between these two extremes, the possibilities for attracting sets of dimension $p$ with $0<p<k$ are much more varied. In the sequel we will focus on the codimension $1$ case.

The key notion in this paper is the following.
\begin{definition}\label{def std}
Let $f$ be an endomorphism of $\Pb^k$ of degree $d$ admitting an attracting set $\Ac$ of dimension $s.$ The endomorphism $f$ is said to be of small topological degree on $\Ac$ if there exists a trapping region $U$ such that
$$\limsup_{n\to+\infty}(\sharp f^{-n}(x)\cap U)^{1/n}<d^s,$$
for any $x\in\Pb^k.$
\end{definition}
Notice that if the inequality above holds for a trapping region $U,$ it also holds for any other trapping region. Sometimes, we will abbreviate this into ``$\Ac$ is an attracting set of small topological degree''.

This notion is stable under small deformations and it ensures that $\Ac$ is not algebraic (cf. \cite{daurat}).

To conclude this section, let us introduce a set of the geometric assumption on $U$ used by Dinh \cite{d-attractor}, which imply good path connectedness properties for $\Cc_p(U).$
We say that an attracting set $\Ac$ with a trapping region $U$ satisfy the assumptions (HD) if the following holds:

\textit{
There exist
an integer $1\leq p \leq k$ and two projective subspaces $I$ and $L$ of
dimension
$p-1$ and $k-p$ respectively such that $I\cap U=\varnothing$ and $L\subset U.$
Since $I\cap L=\varnothing,$ for each $x\in L$ there exists a unique
projective subspace $I(x)$ of dimension $p$ which contains $I$ and such that
$L\cap
I(x)=\{x\}.$ We ask that for each $x\in L,$ the set $U\cap I(x)$ is star-shaped in $x$ as a subset of $I(x)\setminus I\simeq \mathbb C^p.$}

It is easy to check that in this situation $\Ac$ has codimension $p.$

Under these assumptions, Dinh showed in \cite{d-attractor} that continuous forms in $\Cc_p(U)$ equidistribute under push-forward. To be more precise, let $\Lambda$ be the operator on $\Cc_p(U)$ defined by
$$\langle \Lambda S,\phi\rangle:=\langle d^{p-k}f_*S,\phi\rangle:=\langle S, d^{p-k}f^*\phi\rangle.$$
Then,  the following result holds.
\begin{theorem}[\cite{d-attractor}]
Let $f,$ $\Ac$ and $U$ be as above. There exists a current $\tau\in\Cc_p(U)$ such that if $S$ is a continuous form in $\Cc_p(U)$ then
$$\lim_{n\to+\infty}\Lambda^nS=\tau.$$
Moreover $\tau$ is woven, invariant under $\Lambda$ and extremal relative to this last property.
\end{theorem}

\section{Existence of attracting sets of small topological degree in $\Pb^k$}\label{sec ex}
This section is devoted to the proof of Theorem \ref{th ex generic}. We assume that the dimension $k$ is at least $2.$ From now on and in the rest of the paper, we will always denote by $s$ the dimension of our attracting sets, i.e. $s=k-1\geq1.$ Recall from the introduction that if $d\geq2,$ then $\Fc_d$ denotes the set of pairs $f=(f_\infty,R)$ where $R$ is an homogeneous polynomial of degree $d$ in $\Cb^{k}$ and $f_\infty=(F_0,\ldots,F_s)$ is a $k$-tuple of homogeneous polynomials which defines a holomorphic endomorphism of $\Pb^s$ of degree $d.$ By an abuse of notation, we will also denote by $f_\infty$ this endomorphism. Moreover, if $f\in\Fc_d$ and $\varepsilon\in\Cb$  $f_\varepsilon$ stands for the endomorphism of $\Pb^k$ defined by
$$[z_0:\cdots:z_s:z_k]=[z':z_k]\mapsto[f_\infty(z'):z_k^d+\varepsilon R(z')].$$
When $|\varepsilon|$ is small enough, $f_\varepsilon$ admits an attracting set near the hyperplane $\{z_k=0\}.$ Indeed, for $|\varepsilon|\neq 0$ small, we will use the following trapping region
$$U_\varepsilon=\{ [z_0:\cdots:z_s:z_k]\in\Pb^k \, ; \, |z_k|\leq c|\varepsilon| \max(|z_0|,\cdots,|z_s|) \},$$
where $c>0$ is a small number depending on $f\in\Fc_d,$ see \cite[Theorem 4.1 Step 1.]{daurat} for more details. It satisfies the assumptions of Dinh (HD), and we denote by $\Ac_\varepsilon:=\underset{n\in\Nb}{\bigcap}f_\varepsilon^n(U_\varepsilon)$ the associated attracting set.

The proof of Theorem \ref{th ex generic} will be given in  the next two subsections. In the first one, we will define the Zariski open set $\Omega\subset\Fc_d$ and show that if $f\in\Omega$ then $f_\varepsilon$ is of small topological degree on $\Ac_\varepsilon$ for $|\varepsilon|\neq 0$ small enough. This part is just an adaptation of \cite{daurat} to higher dimension. 

The second subsection is devoted to show that $\Omega$ is not empty and therefore, it is a dense  Zariski open set of $\Fc_d.$ Even if it relies on elementary arguments, the construction is quite delicate when $k\geq3.$

\subsection{An algebraic sufficient condition}

All the maps $f_\varepsilon$ preserve the pencil of lines passing through $[0:\cdots:0:1],$ which is naturally parametrized by $\Pb^s.$ Moreover, via this parametrization, the dynamics induced by $f_\varepsilon$ on this pencil is exactly $f_\infty.$ The four following sets of lines play an important role in order to control to number of preimages which are in an attractive region (cf. \cite{daurat}). Denote by

\begin{align*}
&X_{-1}(f):=\{z\in\Pb^s\,;\ \sharp(f_\infty^+)^{-1}(f_\infty^+(z))\geq d^{s-1}+1= d^{k-2}+1\},\\
&X(f):=f_\infty(X_{-1}(f)),\\
&Y_{-2}(f):=\{z\in\Pb^s\,;\, R(z)=0\}\cup\\
&\{z\in\Pb^s\,;\,\exists z'\in\Pb^s, f_\infty^+(z)\neq f_\infty^+(z'),\ f_\infty(z)=f_\infty(z'),f\circ f_\infty^+(z)=f\circ f_\infty^+(z')\},\\
&Y(f):=f_\infty^2(Y_{-2}(f)).
\end{align*}
Here $f_\infty^+$ is the map from $\Pb^s$ to $\Pb^k$ defined by 
$$f_\infty^+(z)=[f_\infty(z):R(z)].$$
Observe that these sets only depend on $f\in\Fc_d,$ not on $\varepsilon.$ When no confusion is possible, we shall just write $X,$ $Y,$ etc.
Notice that in the definition of $X_{-1}(f)$ the  cardinality  is counted with multiplicity and 
$$\left(f_\infty^+\right)^-1\left(f_\infty^+(z)\right) = \{z' \in (f_\infty)^-1(f_\infty(z)) \, | \, R(z')=R(z)\}. $$

The following result gives a sufficient condition on $X$ and $Y$ to ensure the existence of an attracting set with small topological degree. It is analogous to \cite[Proposition 4.3]{daurat}.
\begin{proposition}\label{prop condi nece}
Let $Z(f):=X(f)\cup Y(f).$ If $f\in\Fc_d$ is such that
\begin{itemize}
\item $X(f)$ and $Y(f)$ are disjoint,
\item there exists $n\in\Nb$ such that $\bigcap_{i=0}^nf_\infty^{-i}(Z(f))=\varnothing,$
\end{itemize}
then for $\varepsilon\in\Cb$ with $|\varepsilon|\neq0$ small enough, $f_\varepsilon$ is of small topological degree on $\Ac_\varepsilon.$
\end{proposition}
Denote by $\Omega$ the subset  of pairs $(f_\infty,R)\in \Fc_d$ satisfying the assumptions of this proposition. Before proving the result, let us point out the important fact that the set $\Omega$ is a Zariski open set. Indeed, the sets $X_{-1}$ and $Y_{-2}$ are analytic subsets of $\Pb^s,$ and so the same holds for $X,$ $Y,$ $Z$ and $f_\infty^{-i}(Z),$ for all $i\geq 1.$ When $f$ varies in $\Fc_d,$ it is easy to check that they form analytic subsets of $\Pb^s\times\Fc_d.$ The set $\Omega$ is exactly the complement of the projection on $\Fc_d$ of an analytic subset obtained by intersecting and taking finite union of such sets, see Step 3 of proof of \cite[Theorem 4.1]{daurat} for more details.

\begin{proof}
The proof is essentially the same than the one of \cite[Proposition 4.3]{daurat}. The fact that the pencil of lines is parametrized by $\Pb^s$ instead of $\Pb^1$ doesn't affect the arguments. However, for the convenience of the reader we recall the main ideas.

Throughout the proof, by a line we mean a line passing through $[0:\cdots:0:1]$. We will identify such a line $l$ with the associated point in $\Pb^s,$ i.e. the unique point in the intersection of $l$ with the hyperplane at infinity. The image of $l$ by $f_\varepsilon$ is $f_\infty(l),$ for all $\varepsilon\in\Cb.$ With this identification, the sets $X$ and $Y$ can be seen as sets of points in $\Pb^s$ or sets of lines in $\Pb^k.$

As in \cite[Step 2. Theorem 4.1]{daurat} we reduce the estimation of the number of preimages staying in $f_\varepsilon(U_\varepsilon)$ to a combinatorial problem
by analyzing the geometry of $f_\varepsilon(U_\varepsilon)$, as well as its self-crossings, and the behaviour of the preimages.

A generic line has $d^s$ preimages 
and a generic point $p$ include a line $\ell$ has $d^{s+1}=d^k$ preimages (in $\Pb^k$),  $d$ in each line $\ell_{-1}\in f^{-1}(\ell)$.
So we may :
\begin{itemize}
\item control the number of lines $\ell_{-1}$ that contain preimages staying in $U_\varepsilon$.
\item For each line  $\ell_{-1}$, bound the number of preimages that belong to $f_\varepsilon(U_\varepsilon)$.
\end{itemize}

As in \cite[Step 2. Theorem 4.1]{daurat}, this done by defining  sets $A,B$ of ``good'' lines.

The set $A$ is the set of lines $l$ such that the preimages of a point $x\in f_\varepsilon(U_\varepsilon)\cap l$, 
which lie in $U_\varepsilon$, are contained in at most $d^{s-1}$ lines. Here and in the sequel, everything is counted with multiplicity.

We say that $l$ is in $B$ if for all $x\in l\cap f_\varepsilon^2(U_\varepsilon)$ and each $l_{-1}\in f_\infty^{-1}(l)$  the point $x$ has at most one preimage in $l_{-1}\cap f_\varepsilon(U_\varepsilon)$.

To summarize :
\begin{itemize}
\item If $x\in l\in A$ then there are at most $d^{s-1}$ lines intersecting $f_\varepsilon^{-1}(x)\cap U_{\varepsilon},$ in at most $d$ points for each of them.
\item If $x\in l\in B$ then there are at most $d^{s}$ lines intersecting $f_\varepsilon^{-1}(x)\cap f_\varepsilon(U_{\varepsilon}),$ in at most $1$ point for each of them.
\item If $x\in l\in A\cap B$ then there are at most $d^{s-1}$ lines intersecting $f_\varepsilon^{-1}(x)\cap f_\varepsilon(U_{\varepsilon}),$ in at most $1$ point for each of them.
\end{itemize}
In the two first cases, $f_\varepsilon^{-1}(x)\cap f_\varepsilon(U_{\varepsilon})$ contains at most $d^s$ points. In the last one, it contains at most $d^{s-1}$ points.

By our assumptions, there exists $r>0$ such that
$$X^r \cap Y^r =\varnothing,\ \text{and}\ \bigcap_{i=0}^nf_\infty^{-i}(Z^r)=\varnothing,$$
where $X^r$ (resp. $Y^r,Z^r$) denotes the $r$-neighbourhood of $X$ (resp. $Y,Z$) in $\Pb^s.$ 

Notice that the sets $A$ and $B$ depend on $\varepsilon.$ Indeed, a key point is that the lines in $X$ (resp. $Y$) are exactly those who are not in $A$ (resp. $B$) for any $\varepsilon\in\Cb.$ This fact is illustrated by the following result which can be proved in the same way than \cite[Lemma 4.4]{daurat}.
\begin{lemma}
Fix $r$ as above. Then if $|\varepsilon|\neq 0$ is small enough, we have that 
\begin{equation}
A^c\subset X^r, \, B^c\subset Y^r \\
\text{ and } X\cup Y=Z\subset (A\cap B)^c\subset Z^r.
\end{equation}
\end{lemma}

Now we have reduced the problem to a combinatorial problem. We have to make the same disjunction of cases as in the end of the proof of \cite[Proposition 4.3]{daurat}, except that there are now $n$ step instead of 3. By the previous lemma :
\begin{itemize}
\item If $x\in l \notin X^r$ then $x$ has at most $d$ preimages in at most $d^{s-1}$ lines.
\item If $x\in l \notin Y^r$ then $x$ has at most $1$ preimages in at most $d^{s}$ lines.
\end{itemize}
The condition $X^r \cap Y^r =\varnothing$ imply that 
for all $x\in f^2(U_\varepsilon)$:
\begin{itemize}
\item  either  there are at most $d^{s-1}$ lines intersecting $f_\varepsilon^{-1}(x)\cap U_{\varepsilon},$ in at most $d$ points for each of them,
\item or there are up to $d^{s}$ lines $l_{-1}$ intersecting $f_\varepsilon^{-1}(x)\cap f_\varepsilon(U_{\varepsilon})$, but in each such line there is at most $1$ preimage of $x$.
\end{itemize}
So the number of preimages of a point staying in $f_\varepsilon(U_\varepsilon)$ is less than $d^s$.
Since  $\underset{i=0,..,n}{\bigcap}f_\infty^{-i}(Z^r)=\varnothing$, there exists one step, between $0$ and $n$, such that the preimages of a line are in $A\cap B$ so have less than $d^{s-1}$ preimages. 

This is summarized in the following diagram :
\\
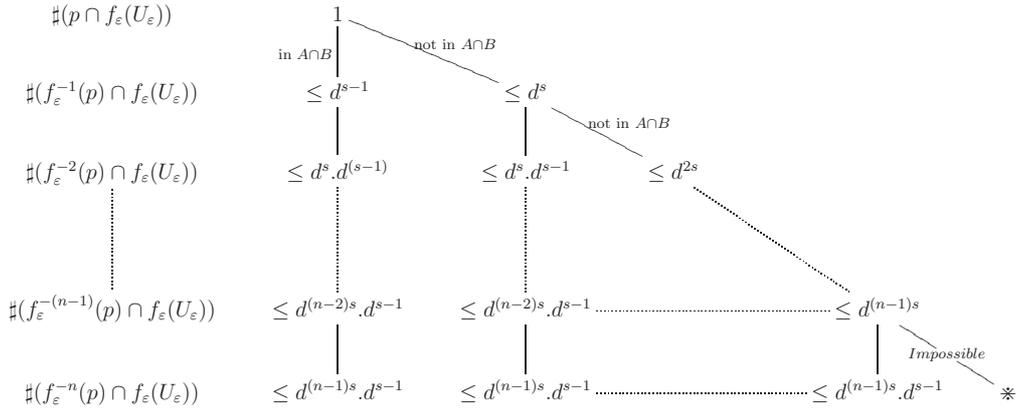
\begin{center}
\scalebox{0.7}{
\xymatrix{
\sharp(p\cap f_\varepsilon(U_\varepsilon))&  1\ar@{-}[d]_{\text{ in } A\cap B}\ar@{-}[rd]^{\text{ not in } A\cap B} & &  &&\\
\sharp(f_\varepsilon^{-1}(p)\cap f_\varepsilon(U_\varepsilon)) &\leq d^{s-1} \ar@{-}[d]&  \leq d^s\ar@{-}[d] \ar@{-}[rd]^{\text{ not in } A\cap B} && \\
 \sharp(f_\varepsilon^{-2}(p)\cap f_\varepsilon(U_\varepsilon))\ar@{.}[dd] & \leq d^s.d^{(s-1)}\ar@{.}[dd] &  \leq d^s.d^{s-1}\ar@{.}[dd] &   \leq d^{2s}  \ar@{.}[rrdd]&\\ 
 & &   &&&\\
\sharp(f_\varepsilon^{-(n-1)}(p)\cap f_\varepsilon(U_\varepsilon)) & \leq d^{(n-2)s}.d^{s-1}\ar@{-}[d]&   \leq d^{(n-2)s}.d^{s-1} \ar@{-}[d] \ar@{.}[rrr] & & &  \leq d^{(n-1)s} \ar@{-}[d] \ar@{-}[rd]|-{Impossible} \\
\sharp(f_\varepsilon^{-n}(p)\cap f_\varepsilon(U_\varepsilon)) & \leq d^{(n-1)s}.d^{s-1}&   \leq d^{(n-1)s}.d^{s-1} \ar@{.}[rrr] & & &  \leq d^{(n-1)s}.d^{s-1} & \divideontimes
 }
}
\\
\end{center}

Thus in each case $ \sharp(f_\varepsilon^{-n}(p)\cap f_\varepsilon(U_\varepsilon)) \leq d^{(n-1)s}.d^{s-1}<d^{ns}$ and we conclude that $f_\varepsilon^n$ is of small topological degree in $f_\varepsilon(U_\varepsilon)$, hence on $\Ac_\varepsilon$ too. This finish the proof.
\end{proof}

\subsection{Construction of examples}
In order to exhibit an element $f\in\Omega\subset\Fc_d$, the first step is to find an example for which $X(f)$ is finite. To this aim, until the end of this section, we only consider  pairs $f=(f_\infty,R)\in\Fc_d$ such that only monomials of the form $z_i^d$ appear in the expression of $f_\infty.$

Let $\xi\in\Cb$ be a primitive $d$-th root of unity. Let $G:=<\xi>\simeq\Zb/d\Zb.$ The group $G^s$ acts on $\Pb^s$ by
$$(\xi_1,\ldots,\xi_s).[z_0:z_1:\cdots:z_s]=[z_0:\xi_1z_1:\cdots:\xi_sz_s].$$
The special form of $f_\infty$ ensures that if $z$ is a preimage of $w\in\Pb^s$ then $f_\infty^{-1}(w)=G^s.\{z\}.$
Therefore, in this case, the set $X_{-1}$ consists of points $z\in\Pb^s$ such that there exists $A\subset G^s$ satisfying $\sharp A=d^{s-1}+1$, $\id\in A$ and $R(g.z)=R(h.z)$ for all $g,h\in A.$ It is easy to see that the $s+1$ points in $\Pb^s$ which have only one non-zero coordinate are fixed by all the elements of $G^s.$ In particular, they are all in $X_{-1}.$ We call these $s+1$ points the \textit{extremal points} of $\Pb^s.$

The following result exhibits an example where $X_{-1},$ and therefore $X,$ is finite.
\begin{proposition}\label{prop X fini}
If $R=(\sum_{i=0}^sz_i)^d$ then $X_{-1}$ is finite.
\end{proposition}
\begin{proof}
We have to prove that if $A\subset G^s$ satisfied $\sharp A=d^{s-1}+1$ then the set $X_A$ of points $z\in\Pb^s$ such that $R(g.z)=R(h.z),$ $\forall g,h\in A,$ is finite.

The choice of $R$ provides two advantages. The first one is that if $R(z)=R(w)$ then there exists $\eta\in\Cb$ such that $\sum_{i=0}^sz_i=\eta\sum_{i=0}^sw_i,$ which is a linear equation. Hence, the problem turns out to be reduce to basic linear algebra. The second one is that the restriction of $R$ to $\{z_i=0\}$ gives the counterpart of $R$ in dimension $s-1$ on $\{z_i=0\}\simeq\Pb^{s-1}.$ It allows us to process by induction on $s.$

For $s=0,$ there is nothing to prove since $\Pb^0$ is reduced to a single point. Now, let $s\geq1$ and assume that the result holds for $s-1.$ To be more specific, assume that for any subset $A'\subset G^{s-1}$ with $\sharp A'=d^{s-2}+1,$ there is only a finite set $X'_{A'}$ of points $z'\in \Pb^{s-1}$ such that $R'(g.z')=R'(h.z')$ for all $g,h\in A',$ where $R'(z'):=(\sum_{i=0}^{s-1}z'_i)^d.$

Let $A\subset G^s$ with $\sharp A=d^{s-1}+1.$ The first step is to show that $X_A\cap\{z_i=0\}$ is finite.

Let $1\leq i\leq s$ and $\alpha\in G.$ If $A_{i,\alpha}$ denotes the set of points $(\xi_1,\ldots,\xi_s)\in A$ such that $\xi_i=\alpha$ then $\sum_{\alpha\in G} \sharp A_{i,\alpha}=d^{s-1}+1.$ Hence, there exists at least one $\tilde \alpha\in G$ such that $\sharp A_{i,\tilde \alpha}\geq d^{s-2}+1.$ Since the projection 
$$\pi\colon(\xi_1,\ldots,\xi_{i-1},\xi_i,\xi_{i+1},\ldots,\xi_s)\mapsto (\xi_1,\ldots,\xi_{i-1},\xi_{i+1},\ldots,\xi_s)$$
 is injective restricted to $A_{i,\tilde \alpha},$ the induction assumption implies, via the natural isomorphisms $\{z_i=0\}\simeq\Pb^{s-1}, X_A\cap\{z_i=0\}\simeq X'_{A'}$ with $A':=\pi(A_{i,\tilde \alpha}),$ that $X_A\cap\{z_i=0\}$ is finite. For $X_A\cap\{z_0=0\}$ the same argument holds with $A':=\pi(A_{i,\tilde \alpha})$ for an arbitrary $1\leq i\leq s.$
 
It remains to prove the result outside $\cup_{i=0}^s\{z_i=0\}.$ Let $z=[z_0:\cdots:z_s]\in X_A$ such that none of its coordinates vanish. In particular, $z_0\neq 0$ so we can normalize by $z_0=1.$ Up to exchange $z$ by $h.z$ and $A$ by $h^{-1}A,$ where $h$ is an arbitrary element of $A,$ we can assume that $\id\in A.$ By definition of $X_A,$ $z$ has to satisfy $R(z)=R(g.z).$ Hence, there exists a function $\eta\colon A\to G$ such that
\begin{equation}\label{eq1}
1+\sum_{i=1}^sz_i=\eta(g)\left(1+\sum_{i=1}^s\xi_iz_i\right),
\end{equation}
for each $g=(\xi_1,\ldots,\xi_s)$ in $A.$

With the same argument as above, there exists $\alpha\in G$ such that the set of $g=(\xi_1,\ldots,\xi_s)\in A$ such that $\eta(g)\xi_s=\alpha$ has more than $d^{s-2}+1$ elements. Up to a change of variables, we can assume that $\alpha=1.$ By induction, we obtain a filtration of $A,$ 
$$A_1\subset\cdots\subset A_{s-1}\subset A_s:=A,$$
such that $\sharp A_i=d^{i-1}+1$ and $\eta(g)\xi_j=1$ if $j>i$ and $g=(\xi_1,\ldots,\xi_s)\in A_i.$

We will now show that there exists $g_i=(\xi_{i,1},\ldots,\xi_{i,s})\in A_i,$ $1\leq i \leq s,$ such that $(z_1,\ldots,z_s)$ is the unique solution of the system of $s$ linear equations obtained from \eqref{eq1}, whose matrix form is
$$\left(\begin{array}{ccccc|c}
\alpha_{1,1}&0&0&\cdots&0&\beta_1\\
\alpha_{2,1}&\alpha_{2,2}&0&\cdots&0&\beta_2\\
\vdots&\vdots &\ddots &\ddots& \vdots&\vdots\\
\alpha_{s-1,1}&\alpha_{s-1,2}&\cdots&\alpha_{s-1,s-1}&0&\beta_{s-1}\\
\alpha_{s,1}&\alpha_{s,2}&\cdots&\cdots&\alpha_{s,s}&\beta_s
\end{array}\right),$$
where $\alpha_{i,j}=\eta(g_i)\xi_{i,j}-1$ and $\beta_i=1-\eta(g_i).$ Indeed, we only have to show that we can find $g_i\in A_i$ such that $\alpha_{i,i}\neq0$ for $1\leq i\leq s.$

For $i=1,$ since our system has at least one solution $(z_1,\ldots,z_s),$ $\alpha_{1,1}=0$ implies that $\beta_1=0$ i.e. $\eta(g_1)=1.$ Hence, $\xi_{1,1}=1$ and $g_1$ has to be equal to $\id.$ But $A_1$ has two distinct elements, so we can choose $g_1\in A_1$ such that $\alpha_{1,1}\neq0.$

Now assume that $\alpha_{j,j}\neq0$ for $1\leq j\leq i-1.$ After simplifications, we obtain the following subsystem
$$\left(\begin{array}{ccccc|c}
\alpha_{1,1}&0&0&\cdots&0&\tilde\beta_1\\
0&\alpha_{2,2}&0&\cdots&0&\tilde\beta_2\\
\vdots&\ddots &\ddots &\ddots& \vdots&\vdots\\
0&\cdots&0&\alpha_{i-1,i-1}&0&\tilde\beta_{i-1}\\
\alpha_{i,1}&\alpha_{i,2}&\cdots&\cdots&\alpha_{i,i}&\beta_i
\end{array}\right),$$
where $\tilde\beta_j\neq 0$ since $z_j\neq 0.$ Since $\alpha_{i,i}=\eta(g_i)\xi_{i,i}-1,$ it vanishes if and only if $\eta(g_i)=\xi_{i,i}^{-1}.$ As above, if it is the case then
$$\beta_i-\sum_{j=1}^{i-1}\alpha_{i,j}\tilde\beta_j=0,$$
which can be rewritten, using $\eta(g_i)=\xi_{i,i}^{-1},$
$$\xi_{i,i}\sum_{j=0}^{i-1}\tilde\beta_j-\sum_{j=0}^{i-1}\xi_{i,j}\tilde\beta_j=0,$$
where $\tilde\beta_0:=1$ and $\xi_{i,0}:=1.$ If $\sum_{j=0}^{i-1}\tilde\beta_j\neq0$ then
$$\xi_{i,i}=\frac{\sum_{j=0}^{i-1}\xi_{i,j}\tilde\beta_j}{\sum_{j=0}^{i-1}\tilde\beta_j},$$
and therefore, for each choice of $\xi_{i,j},$ $1\leq j\leq i-1$ there is a unique possible value for $\xi_{i,i},$ and so, there are at most $d^{i-1}$ choices for which $\alpha_{i,i}=0.$ But $A_i$ has $d^{i-1}+1$ distinct elements, hence there exists $g_i\in A_i$ such that $\alpha_{i,i}\neq 0.$ Finally, if $\sum_{j=0}^{i-1}\tilde\beta_j=0,$ the same argument holds with the equation
$$\xi_{i,i-1}=\frac{\sum_{j=0}^{i-2}\xi_i,j\tilde\beta_j}{\tilde\beta_{i-1}}.$$
Therefore, by induction, the system is invertible and admits a unique solution $(z_1,\ldots,z_s).$

At the end, we have prove that for each function $\eta\colon A\to G,$ there exists at most one point in $X_A\setminus\cup_{i=0}^s\{z_i=0\}.$ In particular, this set is finite.
\end{proof}

\begin{remark}
When $f_\infty$ is of the form as above then $X_{-1}$ has at least $s(d-1)+1$ points counted with multiplicity. We don't know if $X_{-1}$ can be empty in the general case.
\end{remark}
We will need the following result.
\begin{lemma}
Let $Z$ be an analytic subset of $\Pb^s.$ Let $Z_0:=Z$ and $Z_{n+1}:=f_\infty^{-1}(Z_n)\cap Z_n.$ If the analytic subsets $Z_n$ are all non-empty then $Z$ contains a periodic point.
\end{lemma}
\begin{proof}
Observe that $Z_n=\cap_{i=0}^nf_\infty^{-i}(Z).$ Therefore, it is a decreasing sequence of analytic subsets in $\Pb^s$ and so must be stationary, i.e. there exists $n_0\in\Nb$ such that $Z_{n_0}\subset f_\infty^{-1}(Z_{n_0}),$ or equivalently $f_\infty(Z_{n_0})\subset Z_{n_0}.$ Again, the sequence $f_\infty^n(Z_{n_0})$ is decreasing and so there exists $n_1\in\Nb$ such that $f_\infty^{n_1+1}(Z_{n_0})=f_\infty^{n_1}(Z_{n_0}).$ If $Z_{n_0}\neq \varnothing$ then $f_\infty^{n_1}(Z_{n_0})$ is a non-empty analytic subset, invariant by $f_\infty$ so it must contain a periodic point (cf. e.g. \cite[Theorem 5.1]{fak-questions}).
\end{proof}
Using this lemma with $Z=X\cup Y,$ in order to have an example satisfying the assumptions of Proposition \ref{prop condi nece}, it remains to find a map such that $X$ and $Y$ are disjoint and do not contain periodic points. Let $\textrm{Per}_n(f_\infty)$ denote the set of periodic point of $f_\infty$ of period $n.$
\begin{proposition}
If $k\geq2,$ there exists $f=(f_\infty,R)\in\Fc_d$ such that
\begin{itemize}
\item $X\cap Y=\varnothing,$
\item $X\cap \textrm{Per}_n(f_\infty)=\varnothing$ for all $n\in\Nb,$
\item $Y\cap \textrm{Per}_n(f_\infty)=\varnothing$ for all $n\in\Nb.$
\end{itemize}
\end{proposition}
\begin{proof}
We first exhibit an example where $X\cap \textrm{Per}_n(f_\infty)=\varnothing$ for all $n\in\Nb.$ Let $f_\infty$ be an endomorphism of $\Pb^s$ where only monomials of the form $z_i^d$ appear. Moreover, assume that none of the $s+1$ extremal points are preperiodic. It is easy to check that 
$$f_\infty[z_0:\cdots:z_s]=[z_1^d:z_2^d:\cdots:z_s^d:z_s^d+z_0^d]$$
is such a map when $s\geq1.$ By Proposition \ref{prop X fini}, if $R=(\sum_{i=0}^sz_i)^d$ then $X_{-1}(f_\infty,R)$ is finite. Moreover, if $\alpha=(\alpha_0,\ldots,\alpha_s)\in(\Cb^*)^{s+1}$ then $X_{-1}(f_\infty,R_\alpha)$ associated to $R_\alpha=(\sum_{i=0}^s\alpha_i^{-1}z_i)^d$ is the image of $X_{-1}(f_\infty,R)$ by $\alpha$, via the standard action of $(\Cb^*)^{s+1}$ on $\Pb^s$. But, the orbits of this action are all uncountable, except for the extremal points. Hence, by a cardinality argument, we can find $\alpha\in(\Cb^*)^{s+1}$ such that no point in $X_{-1}(f_\infty,R_\alpha)$ is preperiodic. In particular $X(f_\infty,R_\alpha)\cap\textrm{Per}_n(f_\infty)=\varnothing$ for all $n\in\Nb.$

Now, observe that adding $\sum_{i=0}^sc_iz_i^d$ to $R$ doesn't change neither $X$ nor, of course, $\textrm{Per}_n(f_\infty).$ And this operation allows to choose $R$ such that $Y,$ or equivalently $Y_{-2},$ avoids a prescribed countable set. Indeed, let $E\subset \Pb^s$ be a countable set. We can assume that $f_\infty^{-1}(f_\infty(E))=E.$ For $Y_{-2},$ it is clear that, for a very generic choice of $\sum_{i=0}^sc_iz_i^d,$ the resulting $R$ will satisfy $\{R(z)=0\}\cap E=\varnothing.$ For the second part, it is enough to choose $\sum_{i=0}^sc_iz_i^d$ in such a way that $f\circ f_\infty^+(z)\neq f\circ f_\infty^+(z')$ for all $z,z'\in E$ with $z\neq z'.$ Again, it is satisfied for a very generic choice.

Starting from $f=(f_\infty,R_\alpha)$ as above, these arguments with 
$$E=\cup_{n\geq1}\textrm{Per}_n(f_\infty)\cup X$$
 give the desired result.
\end{proof}

\section{Pluripotential theoretic properties and equidistribution}\label{sec pot}
We refer to \cite{de-book} for basics on currents and plurisubharmonic functions and to \cite{ds-lec} for a survey on their use in complex dynamics and in particular for the construction of the Green current $T$ and the equilibrium measure $\mu:=T^k.$

Let $R$ and $S$ be two currents in $\Cc_1(\Pb^k).$ There exists a unique dsh function $u_{R,S},$ i.e. defined outside a pluripolar set of $\Pb^k$ and which is locally the difference of two plurisubharmonic functions, such that
$$\ddc u_{R,S}=R-S\ \ \text{and} \ \ \langle \mu, u_{R,S}\rangle=0.$$
The operator $\Lambda$ also acts on these potentials by
$$\Lambda u_{R,S}(x)=\frac{1}{d^s}\sum_{y\in f^{-1}(x)}u_{R,S}(y),$$
where $s=k-1$ and the points in $f^{-1}(x)$ are counted with multiplicity. It sends bounded (resp. continuous) functions to bounded (resp. continuous) functions and it satisfies $\Lambda \ddc u_{R,S}=\ddc\Lambda u_{R,S}.$

A current $R$ has bounded (resp. continuous) potential if there exists a smooth form $S\in\Cc_1(\Pb^k)$ such that $u_{R,S}$ is bounded (resp. continuous). It implies that $u_{R,S}$ is bounded (resp. continuous) for any current $S\in\Cc_1(\Pb^k)$ with bounded (resp. continuous) potential. The operator $\Lambda$ preserves these two classes of currents. Moreover, $u_{\Lambda R,\Lambda S}=\Lambda u_{R,S}$ since $\ddc$ commutes with $\Lambda$ and
$$\langle \mu,\Lambda u_{R,S}\rangle=\langle d^{-s}f^*\mu,u_{R,S}\rangle=d\langle \mu, u_{R,S}\rangle=0.$$

From now on and for the rest of this section, $f$ is a holomorphic endomorphism of $\Pb^k$ of degree $d\geq2$ which has an attracting set $\Ac$ of codimension $1$ and of small topological degree. Moreover, we assume that $\Ac$ has two trapping regions $U$ and $U'$ such that
$$f(U)\Subset f(U')\Subset U\Subset U'$$
and such that $U$ satisfies (HD).

The main result in this section is that under these assumptions, the attracting current constructed by Dinh has continuous potential.

\begin{theorem}\label{th pot continu}
The attracting current $\tau$ of $\Ac$ has continuous potential. Moreover, there exist constants $A>0$ and $\alpha>0$ such that for all $x,y\in \Pb^k$
$$|u_{\tau,\omega}(x)-u_{\tau,\omega}(y)|\leq A \left|\log(\dist(x,y))\right|^{-\alpha}.$$
\end{theorem}

The proof is based on two facts. The first one is that $\tau$ has bounded potential \cite[Theorem 3.7]{daurat}. The second one is a lemma on the speed of convergence for pluriharmonic form. To be more specific, let $\mathcal H$ be the set of all continuous real $(s,s)$-forms $\phi$ on $U$ such that $\ddc\phi= 0$ and $|\langle R-\tau,\phi\rangle|\leq1$ for all $R\in\Cc_1(U).$ We have the following result \cite[Lemma 4.1]{attrac-speed}, based on the geometry of $\Cc_1(U).$

\begin{lemma}\label{le vitesse pluriharmonique}
There exists a constant $0<\lambda<1$ such that for any $R$ in $\Cc_1(U),$ $\phi$ in $\mathcal H$ and $n$ in $\Nb,$ we have
$$|\langle \Lambda^nR-\tau,\phi\rangle|\leq\lambda^n.$$
\end{lemma}
It has the following corollaries.
\begin{corollary}\label{cor vitesse pluriharmonique}
There exist constants $C>0$ and $0<\lambda<1$ such that if $R\in\Cc_1(U)$ is a current with bounded potential and $x$ is a point outside $U'$ then
$$|u_{\Lambda^nR,\tau}(x)|\leq C\lambda^n.$$
\end{corollary}
\begin{proof}
If $x\notin U'$ then there exists a real $(s,s)$-current $\phi_x$ which is smooth on $U'$ and such that $\ddc\phi_x=\delta_x-\mu.$ Moreover, there exists a constant $C>0,$ independent of $x,$ such that $|\langle R-\tau,\phi_x\rangle|\leq C$ for all $x$ outside $U'.$ Indeed,
\begin{equation}\label{eq formel}
\langle R-\tau,\phi_x\rangle=\langle \ddc u_{R,\tau},\phi_x\rangle=\langle u_{R,\tau},\ddc\phi_x\rangle=u_{R,\tau}(x)
\end{equation}
and $\{u_{R,\tau}\}_{R\in\Cc_1(\overline U)}$ is a compact family of pluriharmonic functions on $\Pb^k\setminus \overline U$ and then, it is uniformly bounded on $\Pb^k\setminus U'.$ Therefore, the restriction of $\phi_x/C$ to $U$ is a form in $\mathcal H$ and Lemma \ref{le vitesse pluriharmonique} implies that
$$|u_{\Lambda^nR,\tau}(x)|=|\langle \Lambda^nR-\tau,\phi_x\rangle|\leq C\lambda^n.$$
\end{proof}

\begin{remark}
The second equality in \eqref{eq formel} is a priori just a formal equality since both sides in the brackets involve currents. Anyway, in our case the singular locus of these currents are disjoint so it is easy, using regularization, to give a rigorous meaning to this equality. In the general case, it is related to the theory of super-potential of Dinh and Sibony.
\end{remark}

\begin{proof}[Proof of Theorem \ref{th pot continu}]
If $S$ is a current in $\Cc_1(U)$ with continuous potential then it is also the case for $\Lambda^nS,$ $\forall n\in\Nb.$ Hence $(u_{S,\Lambda^nS})_{n\in\Nb}$ is a sequence of continuous functions. We will show that it converges uniformly to $u_{S,\tau}$ which will be therefore continuous. Indeed, as $u_{S,\tau}=u_{S,\Lambda^nS}+u_{\Lambda^nS,\tau},$ $\forall n\in \Nb,$ we have to prove that $u_n:=u_{\Lambda^nS,\tau}$ converge uniformly to $0.$ Observe that $u_n=\Lambda^nu_0.$

By the definition of $\Lambda$ acting on functions, we have
$$|u_{n+1}(x)|=|\Lambda u_n(x)|=\frac{1}{d^{s}}\left|\sum_{y\in f^{-1}(x)\cap U'}u_n(y)+\sum_{y\in f^{-1}(x)\cap \Pb^k\setminus U'}u_n(y)\right|.$$
Since $f$ is of small topological degree on $\Ac,$ we can assume, possibly by exchanging $f$ by an iterate $f^N,$ that $f^{-1}(x)\cap U'$ has strictly less that $d^{s}$ points, $\forall x\in\Pb^k.$ This and Corollary \ref{cor vitesse pluriharmonique} give
$$|u_{n+1}(x)|\leq\frac{d^{s}-1}{d^{s}}\|u_n\|_\infty+\frac{d^k}{d^{s}}C\lambda^n.$$
Without loss of generality, we can assume that $(d^{s}-1)/d^{s}\leq\lambda<1$ and then obtain by induction that
$$\|u_{n+1}\|_\infty\leq \lambda^{n}(\|u_0\|_\infty+(n+1)dC).$$
As $\tau$ has bounded potential \cite{daurat}, $\|u_0\|_\infty$ is finite and therefore $(\|u_n\|_\infty)_{n\in\Nb}$ converges exponentially to zero. This finish the proof of the continuity of the potential of $\tau.$

To estimate the modulus of continuity, we assume now that $S\in\Cc_1(U)$ is a current which has Lipschitz continuous potential and such that the same holds for $\Lambda S.$ Up to replace $U$ by $f^{-1}(f(U)),$ we can choose $S=d^{-1}f^*R$ where $R\in\Cc_1(f(U))$ is smooth. If $v_n:=u_{\Lambda^nS,S}$ then it easy to check that $v_{n+1}=\Lambda v_n+v_1,$ $v_n-u_1=u_n$ and that $v_1$ is a Lipschitz continuous function. Our goal is to estimate the modulus of continuity of $u_1=u_{\tau,S}.$

By a result of Dinh and Sibony, see \cite[Lemma 4.3]{dinh:sib:equigreen}, there exist constants $r>0$ and $C_0>1$ such that for all $x,y\in \Pb^k$ if $\dist(x,y)<r$ then it is possible to order the sets $f^{-1}(x)=\{x_i\}_{1\leq i\leq d^k}$ and $f^{-1}(y)=\{y_i\}_{1\leq i\leq d^k}$ in such a way that
$$\dist(x_i,y_i)\leq C_0\dist(x,y)^{d^{-k}}.$$
It implies that there exists a constant $C_1>0$ such that for all $x,y\in \Pb^k$  and  all $n\in \Nb$
$$|v_n(x)-v_n(y)|\leq nd^{n} C_0^{\sum_{i=0}^{n-1}d^{-ik}}\dist(x,y)^{d^{-kn}}\leq nd^{n} C_1\dist(x,y)^{d^{-kn}}.$$
Indeed, this is clear for $n=1$ as $v_1$ is Lipschitz with a constant which can be assumed smaller than $C_0$ and $C_1.$ And if the statement holds for $v_n$ then if $\dist(x,y)<r$ we have
\begin{align*}|v_{n+1}(x)-v_{n+1}(y)|&\leq|\Lambda v_n(x)-\Lambda v_n(y)|+|v_1(x)-v_1(y)|\\
&\leq d^{-s}\sum_{j=1}^{d^k}\left|v_n(x_j)-v_n(y_j)\right|+C_1\dist(x,y)\\
&\leq d^{-s}\sum_{j=1}^{d^k}nd^{n}C_0^{\sum_{i=0}^{n-1}d^{-ik}}\dist(x_j,y_j)^{d^{-kn}}+C_1\dist(x,y)\\
&\leq nd^{n+1} C_0^{\sum_{i=0}^{n}d^{-ik}}\dist(x,y)^{d^{-k(n+1)}}+C_1\dist(x,y)\\
&\leq (n+1)d^{n+1} C_0^{\sum_{i=0}^{n}d^{-ik}}\dist(x,y)^{d^{-k(n+1)}}.
\end{align*}
The case where $\dist(x,y)\geq r$ comes from the fact that the functions $v_n=u_1+u_n$ are uniformly bounded since $u_1$ is continuous and $\|u_n\|_\infty\leq C\lambda^n.$

We deduce that for all $x,y\in \Pb^k$ 
\begin{align*}
|u_1(x)-u_1(y)|&\leq \underset{n\in \N}{\inf} \left( 2\|u_1-v_n\|_\infty+|v_n(x)-v_n(y)| \right)\\
&\leq \underset{n\in \N}{\inf} \left( 2 C\lambda^n +  nd^{n} C_1\dist(x,y)^{d^{-kn}} \right).
\end{align*}
Let  $x,y$ be in $\Pb^k$ such that $\dist(x,y)<<1$ then, if
$$n=\left\lfloor\dfrac{-\log (\sqrt{-\log(\dist(x,y))})}{\log(1/d^k)}\right\rfloor=\left\lfloor\dfrac{\log (-\log(\dist(x,y)))}{2k\log(d)}\right\rfloor,$$
we have 
$${d^{kn}}=e^{nk\log(d)}\leq e^{\frac{\log (-\log(\dist(x,y)))}{2k\log(d)}\times k\log(d)}= e^{\frac{1}{2}\log (-\log(\dist(x,y)))}=(-\log(\dist(x,y)))^{1/2}$$
so 
$$nd^{n}\dist(x,y)^{1/d^{kn}}\leq d^{kn}e^{\frac{\log(\dist(x,y))}{d^{kn}}}\leq \sqrt{-\log(\dist(x,y))}e^{-\sqrt{-\log(\dist(x,y))}}$$
and 
$$\lambda^n \leq e^{\log(\lambda)\left( \frac{\log (-\log(\dist(x,y)))}{2k\log(d)} -1 \right)}
\leq  \dfrac{(-\log\dist(x,y))^\frac{\log(\lambda)}{2k\log(d)}}{\lambda}.$$
If $\dist(x,y)$ is small enough then $$\sqrt{-\log(\dist(x,y))}e^{-\sqrt{-\log(\dist(x,y))}}\leq (-\log\dist(x,y))^\frac{\log(\lambda)}{2k\log(d)}.$$ As $u_1$ is bounded \cite{daurat}, there exists $A>0$ such that for all $x,y\in \Pb^k$
$$|u_1(x)-u_1(y)|\leq A \left|\,\log\dist(x,y) \,\right|^{-\alpha}$$
where $\alpha=\frac{-\log(\lambda)}{2k\log(d)}$.
This finish the proof of Theorem \ref{th pot continu}.
\end{proof}
\begin{remark}
Without (HD), similar arguments, using
$$|u_{\Lambda^nR,\tau}(x)|\leq C\ \ \ \forall x\not\in U'$$
instead of Corollary \ref{cor vitesse pluriharmonique}, give an alternative proof to the fact that $\tau$ has bounded potential.
\end{remark}
We can deduce from the proof of Theorem \ref{th pot continu} that under the same assumptions $\tau$ is the unique current in $\Cc_1(U)$ invariant by $\Lambda$ and with bounded potential. In fact, the boundedness assumption can be dropped and moreover, Theorem \ref{th speed} gives a uniform speed of equidistribution of any current in $\Cc_1(U)$ towards $\tau.$
\begin{proof}[Proof of Theorem \ref{th speed}]
Let $S$ be a current in $\Cc_1(U).$ The functions $u_n:=u_{\Lambda^nS,\tau}$ satisfy $\Lambda u_n=u_{n+1}.$ Observe that if $\phi$ is a $\mathcal C^2$ test form then
$$|\langle \Lambda^n S-\tau,\phi\rangle|=|\langle u_n,\ddc \phi\rangle|\leq \|u_n\|_{L^1(\Pb^k)}\|\phi\|_{\mathcal C^2}.$$
So, in order to prove the theorem, it is sufficient to obtain an estimate on $\|u_n\|_{L^1(\Pb^k)}.$

We now introduce some notations. First, notice that, as the constants below will not depend on $S,$ considering an iterate $f^N$ instead of $f$ does not affect the arguments. As in the proof of Theorem \ref{th pot continu} we can assume, possibly by exchanging $f$ by $f^N,$ that $f^{-1}(x)\cap U$ has strictly less that $d^s$ points, for all $x\in\Pb^k.$ Let $C$ and $\lambda$ be the constants of Corollary \ref{cor vitesse pluriharmonique} where, as above, we can assume that $(d^s-1)/d^s<\lambda<1.$ Finally, let $\alpha>0$ and $\delta>0$ be such that $\lambda<\delta<1$ and $(\alpha-1)\delta^n>Cdn\lambda^{n-1},$ for all $n\geq 0.$ 

Since
$$|u_{n+1}(x)|=|\Lambda u_n(x)|=\frac{1}{d^{s}}\left|\sum_{y\in f^{-1}(x)\cap U'}u_n(y)+\sum_{y\in f^{-1}(x)\cap \Pb^k\setminus U'}u_n(y)\right|,$$
in the same way than in the proof of Theorem \ref{th pot conti}, we obtain that
$$|u_{n+1}(x)|\leq \lambda \max_{y\in f^{-1}(x)}|u_n(y)|+Cd\lambda^n,$$
and by induction
$$|u_{n+1}(x)|\leq \lambda^{n+1}\max_{y\in f^{-(n+1)}(x)}|u_0(y)|+Cd(n+1)\lambda^n.$$
Therefore
$$\{|u_n|\geq \alpha\delta^n\}\subset f^n\left(\left\{|u_0|\geq \frac{\alpha\delta^n-Cdn\lambda^{n-1}}{\lambda^n}\right\}\right)\subset f^n\left(\left\{|u_0|\geq \frac{\delta^n}{\lambda^n}\right\}\right),$$
where the last inclusion comes from the choices of $\alpha$ and $\delta.$

On the other hand, the family $\Pc$ of functions $u$ such that $\langle \mu,u\rangle=0,$ $\ddc u=S-\tau$ with $S\in \Cc_1(\mathbb P^k)$ is a compact family of dsh functions. Therefore, by an inequality of Hörmander (it can be deduced from \cite[Theorem 4.4.5]{ho-book}), there exists a constant $a>0$ such that
$$\int_{\Pb^k}e^{a|u|}\omega^k\leq a^{-1},$$
for any $u\in\Pc.$ As a consequence of this inequality, we have
$$\textrm{Vol}\left(\left\{|u_0|\geq\frac{\delta^n}{\lambda^n}\right\}\right)\leq a^{-1}e^{-a\delta^{n}\lambda^{-n}},$$
since $u_0\in\Pc.$ And, by the change of variables formula, there exists a constant $M>0,$ which is essentially equal to $\|Jac(f)\|_\infty,$ such that 
$$\textrm{Vol}\left(f^n\left(\left\{|u_0|\geq\frac{\delta^n}{\lambda^n}\right\}\right)\right)\leq a^{-1}M^ne^{-a\delta^{n}\lambda^{-n}},$$
which implies
$$\textrm{Vol}\left(\{|u_n|\geq \alpha\delta^n\}\right)\leq a^{-1}M^ne^{-a\delta^{n}\lambda^{-n}},$$
as $\{|u_n|\geq \alpha\delta^n\}\subset f^n\left(\left\{|u_0|\geq \frac{\delta^n}{\lambda^n}\right\}\right).$

Hence, by Cauchy-Schwarz inequality
$$\|u_n\|_{L^1(\Pb^k)}\leq \textrm{Vol}(\{|u_n|\geq \alpha\delta^n\})^{1/2}\|u_n\|_{L^2(\Pb^k)}+\alpha\delta^n,$$
and then
$$\|u_n\|_{L^1(\Pb^k)}\leq a^{-1/2}M^{n/2}e^{-a\delta^{n}\lambda^{-n}/2}\|u_n\|_{L^2(\Pb^k)}+\alpha\delta^n\leq c\delta^n,$$
for a constant $c>0$ large enough since $\|u\|_{L^2(\Pb^k)}$ is uniformly bounded for $u\in\Pc$ and $\lambda<\delta.$
\end{proof}

\begin{remark}
Using interpolation between Banach spaces, see e.g. \cite{triebel-interpolation}, we can obtain that
$$|\langle \Lambda^nS-\tau,\phi\rangle|\leq c\delta^{\alpha n/2}\|\phi\|_{\mathcal C^\alpha},$$
for any test form $\phi$ of class $\mathcal C^\alpha,$ with $0<\alpha\leq2.$ Here, $c$ depends on $\alpha.$
\end{remark}

\section{Negative Lyapunov exponent} \label{section expo str négatif}
In this section, we consider an endomorphism $f$ of $\Pb^k$ of degree $d\geq 2$ with an attracting set $\Ac$ of codimension $1.$ We will study hyperbolicity properties of measures of entropy $s\log d$ on $\Ac.$

Let $\nu$ be such a measure and assume that $\nu$ is ergodic. We will use the ideas of de Thélin \cite{deT:expo} to obtain a relation between the Lyapunov exponents of $\nu$ and a certain ``topological degree''. Indeed, as the dynamics near an attracting set stay confined in a trapping region, we consider the following local version of topological degree. Let $V$ be a small neighborhood of $\supp(\nu)$ and define
$$d_{k,loc}:=\limsup_{n\to+\infty}\left(\int_V(f^n)^*(\omega^k)\right)^{1/n}.$$
We will now prove Theorem \ref{th hyp} which says that if $d_{k,loc}<d^s$ then $\nu$ is a hyperbolic measure.
\begin{remark}
As we said in the introduction, to be of small topological degree on $\Ac$ is much stronger than $d_{k,loc}<d^{s}$. Moreover, the assumption $d_{k,loc}<d^{s}$ holds in all examples known by the authors. We know neither if this assumption nor the conclusion on the hyperbolicity of $\nu$ always holds.

\end{remark}

\begin{proof}[Proof of Theorem \ref{th hyp}]
If $\log(\Jac f)\notin L^1(\nu)$ then $\chi_k=-\infty$ and we have the upper bound one $\chi_k$.
By Margulis-Ruelle inequality, we have $\chi_1>0$.
Thus $\chi_1\neq \chi_k$ and the lower bound $\chi_{k-1}\geq 2^{-1}\log d$ is a direct consequence of \cite[Theorem D]{dupont:entropy}.

So we assume that $\log(\Jac f)\in L^1(\nu)$. By  \cite{deT:expo}, we have $\chi_{k-1} >0$ and the lower bound $\chi_{k-1}\geq 2^{-1}\log d$ as long as we know that $\chi_k\leq 0.$

The proof of the upper bound on $\chi_k,$ follows closely the one of \cite{deT:expo}. We will recall the technical intermediary results he obtains 
and will only give all details when the assumption about the dynamical degree $d_{k,loc}$ appears.
 
 Actually, everything in the sequel, as in the proofs of \cite{deT:expo}, will take place in a fixed compact set $K\subset V.$

Let $\gamma>0$ be an arbitrary small constant. If $\gamma$ is small enough, there exist a compact $K\subset V$ containing a $\gamma$-neighboorhood of $\supp(\nu)$. Hereafter, $\delta>0$ will be a small constant depending on $\gamma.$ Recall that two points $x$ and $y$ in $\Pb^k$ are $(n,\delta)$-separated if $\dist_n(x,y)\geq\delta$ where $\dist_n(x,y):=\max_{0\leq i\leq n}\dist(f^i(x),f^i(y)).$ Using a theorem of Brin and Katok \cite{BK}, de Thélin constructs, for $n$ large enough, a set $E=\{x_i\}_{1\leq i\leq N}\subset \supp(\nu)$ 
such that
\begin{itemize}
\item $N\geq \frac{1}{2}e^{h(\nu)n-\gamma n},$
\item the points $x_i$ and $x_j$ are $(n,\delta)$-separated if $i\neq j.$
\end{itemize}
Then for each $x_i\in E,$ using Pesin theory and graph transform, he constructs an
 ``thick unstable manifold'' $W(x_i)\subset K,$ which in our case is an open subset, such that $f^n(x_i)\in W(x_i)$ and $\textrm{Vol}(W(x_i))\geq c^{-1}e^{2\chi_k^-n-c\gamma n}$ for some constant $c\geq 1$ independent of $n.$
With the terms of \cite{BLS}, $W(x_i)$ is the $n$th cut-off iterate of a Lyapunov chart.

Moreover, we can deduce from his construction that if $x_i\in E$ then there exists an open set $W_{-n}(x_i)\subset K$ which contains $x_i$ and such that 
\begin{itemize}
\item $f^n(W_{-n}(x_i))=W(x_i),$
\item the diameter of $f^l(W_{-n}(x_i))$ is less than $\delta/4$ for $l=0,\ldots,n.$
\end{itemize}
These two properties and the fact that the points $x_i$ are $(n,\delta)$-separated imply that the sets $W_{-n}(x_i)$ are pairwise disjoint. In particular, if, for $a\in K$, $n(a)$ denotes the number of points $x_i\in E$ such that $a\in W(x_i),$ then there exist at least $n(a)$ points in $f^{-n}(a)\cap V,$ i.e. $n(a)\leq (f^n)_*\textbf{1}_V(a).$

We can now use our assumption that $d_{k,loc}<d^s.$ The last observation implies that
\begin{equation}\label{eq2}
\int_K n(a)\omega^k\leq\int (f^n)_*(\textbf{1}_V)\omega^k=\int_V(f^n)^*(\omega^k).
\end{equation}
Moreover, by the definition of $d_{k,loc},$ for all $\varepsilon>0,$ we have for $n$ large enough
\begin{equation}\label{eq3}
\int_V(f^n)^*(\omega^k)\leq (d_{k,loc}+\varepsilon)^n.
\end{equation}
On the other hand, the estimates on $N$ and on $\textrm{Vol}(W(x_i))$ imply that
\begin{equation}\label{eq4}
\int_K n(a)\omega^k\geq \frac{1}{2c}e^{h(\nu)n+2\chi_k^--(c+1)\gamma n}.
\end{equation}
Hence, we deduce from the equations \eqref{eq2}, \eqref{eq3} and \eqref{eq4} that
$$(d_{k,loc}+\varepsilon)^n\geq \frac{1}{2c}e^{h(\nu)n+2\chi_k^--(c+1)\gamma n},$$
for $n$ large enough. As the constant $c$ is independent of $n,$ we obtain, when $n$ goes to infinity, that $h(\nu)\leq \log(d_{k,loc}+\varepsilon)+(c+1)\gamma-2\chi_k^-.$ But $\varepsilon>0$ and $\gamma>0$ are arbitrary small, hence $h(\nu)\leq \log(d_{k,loc})-2\chi_k^-,$ which was the desired result.
\end{proof}

\bibliographystyle{alpha}

\noindent
Université Paris-Est Marne-La vallée, LAMA, UMR 8050,
77454 Marne-la-Vallée Cedex 2, France.\\
{\tt  e-mail: Sandrine.Daurat-goguet@u-pem.fr}
\vspace{0,5cm}
\\
Université de Bourgogne, IMB, UMR 5584, 21078 Dijon Cedex, France.\\
{\tt e-mail: johan.taflin@u-bourgogne.fr}
\end{document}